\UseRawInputEncoding

\documentclass[11pt,leqno,a4paper]{amsart}
\usepackage{amssymb,amscd,amsmath}
\usepackage{pspicture}
\usepackage{graphicx}

\usepackage{mathptmx}

\usepackage[matrix,arrow,curve,frame]{xy}    

\xymatrixcolsep{1.9pc}                          
\xymatrixrowsep{1.9pc}
\newdir{ >}{{}*!/-5pt/\dir{>}}                  

 \calclayout

\makeatletter

\renewcommand{\subsection}{\@startsection{subsection}{1}{0pt}{-3.25ex plus -1ex minus-.2ex}{1.5ex plus.2ex}{\normalfont\it}}
\renewcommand{\section}{\@startsection{section}{1}{\parindent}{3.5ex plus 1ex minus .2ex}{2.3ex plus.2ex}{\sc}}

\renewcommand{\phi}{\varphi}
\renewcommand{\leq}{\leqslant}
\renewcommand{\geq}{\geqslant}
\renewcommand{\epsilon}{\varepsilon}

\renewcommand{\kappa}{\varkappa}

\DeclareMathOperator{\spec}{Spec}
 \DeclareMathOperator{\cyl}{cyl}

 \DeclareMathOperator{\mot}{mot}

 \DeclareMathOperator{\cone}{cone}
\DeclareMathOperator{\Hom}{Hom}

 \DeclareMathOperator{\colim}{colim}

 \DeclareMathOperator{\nis}{nis}

 \DeclareMathOperator{\Mod}{Mod}
 \DeclareMathOperator{\Ob}{Ob}

\newcommand{\cc}{\mathcal}
\newcommand{\bb}{\mathbb}

\newcommand{\op}{{\textrm{\rm op}}}

\newcommand{\sheff}{SH^{\mot}_{S^1}}
\newcommand{\shnis}{SH^{\nis}_{S^1}}

\newtheorem{thm}{Theorem}[section]

\newtheorem{cor}[thm]{Corollary}
\newtheorem{lem}[thm]{Lemma}

\newtheorem{rem}[thm]{Remark}

\newtheorem{defs}[thm]{Definition}

\makeatother

\begin{document}

\footskip30pt


\title{On the triangulated category of framed motives $\text{DFr}_{-}^{eff}(k)$}

\author{Ivan Panin}
\address{St. Petersburg Branch of V. A. Steklov Mathematical Institute,
Fontanka 27, 191023 St. Petersburg, Russia}
\email{paniniv@gmail.com}

\thanks{
}

\keywords{Motivic homotopy theory, framed correspondences, spectral
categories}

\subjclass[2010]{14F42, 19E08, 55U35}

\begin{abstract}
The category of framed correspondences $Fr_*(k)$ was invented by
Voevodsky \cite[Section 2]{Voe2} in order to give 
another framework for $\text{SH}(k)$ more amenable
to explicit calculations. 
Based on \cite{Voe2} and \cite{GP4} Garkusha and the author introduced in 
\cite[Section 2]{GP5}
a triangulated category of framed bispectra $\text{SH}_{nis}^{fr}(k)$. 
It is shown in \cite[Section 2]{GP5} that $\text{SH}_{nis}^{fr}(k)$ recover classical Morel--Voevodsky triangulated
categories of bispectra $\text{SH}(k)$.

For any infinite perfect field $k$ a triangulated category of $\bb {F}\text{r}$-motives
$\text{D}\bb {F}\text{r}_{-}^{eff}(k)$ is constructed in the style of Voevodsky's
construction of the category $\text{DM}_-^{eff}(k)$. 
In our approach the Voevodsky category of Nisnevich sheaves with transfers is replaced
with the category of $\bb {F}\text{r}$-modules.
To each smooth
$k$-variety $X$ the $\bb {F}\text{r}$-motive $\text{M}_{\bb {F}\text{r}}(X)$ is associated in the
category $\text{D}\bb {F}\text{r}_{-}^{eff}(k)$.

We identify the triangulated category $\text{D}\bb {F}\text{r}_{-}^{eff}(k)$ with the full triangulated subcategory $\text{SH}^{eff}_{-}(k)$
of the classical Morel--Voevodsky triangulated category $\text{SH}^{eff}(k)$ of effective motivic bispectra
\cite{Jar2}. Moreover, the triangulated category $\text{D}\bb {F}\text{r}_{-}^{eff}(k)$ is naturally 
{\it symmetric monoidal}. 
Particularly, 
$\text{M}_{\bb {F}\text{r}}(X)\otimes_{\bb {F}\text{r}} \text{M}_{\bb {F}\text{r}}(Y)=\text{M}_{\bb {F}\text{r}}(X\times Y)$.
The mentioned
identification of the triangulated categories respects the symmetric monoidal structures on both sides.

We work with the derived category $\text{D}\bb {F}\text{r}_-(k)$ of bounded below $\bb {F}\text{r}$-modules
rather than with the homotopy category $\text{SH}_{nis}(k)$ of bispectra as in \cite[Section 2]{GP5}.
\end{abstract}
\maketitle

\thispagestyle{empty} \pagestyle{plain}

\newdir{ >}{{}*!/-6pt/@{>}} 


\section{Introduction}

The Voevodsky triangulated category of motives
$\text{DM}_-^{eff}(k)$~\cite{Voe1} provides a natural framework to study
motivic cohomology.
In this paper a new short approach to constructing the part $\text{SH}^{eff}_{-}(k)$
of the classical triangulated
category $\text{SH}(k)$ is presented providing the base field is infinite and perfect.

We work in the framework of
strict $V$-spectral categories introduced in
\cite[Definition~\ref{vsp}]{GP2}
The main new feature of our  spectral category $\bb {F}\text{r}$
is that {\it it is symmetric monoidal}. It is also connective and Nisnevich excisive in the sense
of~\cite{GP}. 
Each $\pi_0(\bb {F}\text{r})$-presheaf $\cc F$ of Abelian groups 
is automatically a radditive framed presheaf of Abelian groups
in the sense of \cite{Voe2}. By \cite[Lemma 2.15]{GP3} such an $\cc F$ 
is a $\bb ZF_*(k)$-presheaf of Abelian groups 
in the sense of \cite[2.13]{GP3}. 
By \cite[Lemma 4.5]{Voe2} and \cite[Lemma 2.15]{GP3} its associated Nisnevich sheaf
$\cc F_{nis}$ is canonically a $\bb ZF_*(k)$-presheaf of Abelian groups. 
If $\cc F$ is homotopy invariant and stable in the sense of \cite{Voe2} 
(see also \cite[Def. 2.13, 2.14]{GP3}), then by 
\cite[Thm. 1.1]{GP3} the framed Nisnevich sheaf
$\cc F_{nis}$ is strictly homotopy invariant and stable.

The main symmetric monoidal strict $V$-spectral category $\bb {F}\text{r}$ is constructed in
Section \ref{The_Category}. It is strict over infinite perfect fields. 
Denote by $\text{D}\bb {F}\text{r}_-(k)$ the full triangulated subcategory of
$\text{SH}^{nis}(\bb {F}\text{r})$ of bounded below $\bb {F}\text{r}$-modules. We also denote by
$\text{D}\bb {F}\text{r}_-^{eff}(k)$
the full triangulated subcategory of
$\text{D}\bb {F}\text{r}_-(k)$
of those $\bb {F}\text{r}$-modules $M$ such that each
$\bb ZF_*(k)$-presheaf
$\pi_i(M)|_{\bb ZF_*(k)}$ is 
{\it homotopy invariant and stable} 
in the sense of
\cite[Def. 2.13, 2.14]{GP3}. 

We call $\text{D}\bb {F}\text{r}_-^{eff}(k)$
{\it the triangulated category of $\bb {F}\text{r}$-motives}.
The category $\text{D}\bb {F}\text{r}_{-}^{eff}(k)$ is naturally symmetric monoidal. 
For each 
$X\in Sm/k$ the $\bb {F}\text{r}$-module
$$C_*(\bb {F}\text{r}(X)):=|d\mapsto\underline{\Hom}(\Delta^d,\bb {F}\text{r}(X))|$$
belongs to $\text{D}\bb {F}\text{r}_{-}^{eff}(k)$ and is called 
{\it the} $\bb {F}\text{r}$-{\it motive of} $X$; \ 
$\text{M}_{\bb {F}\text{r}}(X)\otimes_{\bb {F}\text{r}} \text{M}_{\bb {F}\text{r}}(Y)=\text{M}_{\bb {F}\text{r}}(X\times Y)$.

The latter triangulated category {\it is identified} with the full triangulated subcategory $\text{SH}^{eff}_{-}(k)$
of the classical Morel--Voevodsky triangulated category $SH^{eff}(k)$ of effective motivic bispectra
({\it this is the main result of the preprint}). See Theorem \ref{VeryMain}.

The mentioned
identification respects {\it the symmetric monoidal structures} on both sides.

It can be shown that the identification triangulated functor as in Theorem \ref{VeryMain}
$$\bb M_{\text{SH}}: \text{D}\bb {F}\text{r}_-^{eff}(k)\to SH^{eff}_-(k)$$
takes the $\bb {F}\text{r}$-motive
$\text{M}_{\bb {F}\text{r}}(X)$ of $X$ to
the symmetric bispectrum
$\Sigma_{\bb G_m}\Sigma_{S^1}(X_+)$. 

Sections 2 and 3 contains the materials of 
\cite[Sections 2 and 3]{GP2}
adapted to the symmetric monoidal spectral category
$\bb {F}\text{r}$,
which is defined in Section 4. 
In Section 4 the language of triangulated categories is used 
as opposed to the model categories language. This allows to state all constructions
and results in a very explicit form. 
The main result here is Theorem \ref{neploho}. 
However it seems that this language does not allow to prove 
Theorem 6.2 (the main result of this preprint).

Also this language does not allow 
to state and prove the following true result: 
there is a triangulated equivalence of the triangulated categories
$$\text{SH}^{mot}(\bb {F}\text{r})\to \text{SH}^{mot}(k).$$ 

Triangulated subcategories $\text{SH}^{nis}(\bb {F}\text{r})$,
$\text{D}\bb {F}\text{r}_-(k)$ 
and 
$\text{D}\bb {F}\text{r}_-^{eff}(k)$
are defined in Section 5. 
The main result of the preprint (Theorem \ref{VeryMain})
is stated in Section 6. 
Its proof is postponed to the next preprint.

Throughout the paper we denote by $Sm/k$ the category of smooth
separated schemes of finite type over the base field $k$. The base field $k$
is supposed to be infinite and perfect. The paper 
\cite{DP} shows that there is no restriction on the characteristic of $k$. \\
{\bf Acknowledgements}. The author is very grateful to G.Garkusha for his deep interest 
in the topic of this preprint. I am very grateful also to my mother in law 
K.Shahbazian for 
her very stimulating interest to the present work on all its stages.

\section{Preliminaries}

We work in the framework of spectral categories and modules over
them in the sense of Schwede--Shipley~\cite{SS}. We start with
preparations. 

We follow \cite[Definition 2.1.1, Remark 2.1.5]{HSS}.
A symmetric sequence of objects in a category $\cc C$ is a functor $\Sigma \to \cc C$, and
the category of symmetric sequences of objects in $\cc C$ is the functor category $\cc C^{\Sigma}$. 
The category $\Sigma$ is a skeleton of the category of finite sets and
isomorphisms. Hence every symmetric sequence has an extension, which is unique
up to isomorphism, to a functor on the category of all finite sets and isomorphisms. 
We will use both view points (often the second one).

Recall that symmetric spectra have two sorts of homotopy groups
which we shall refer to as {\it naive\/} and {\it true homotopy
groups\/} respectively following terminology of~\cite{Sch}.
Precisely, the $k$th naive homotopy group of a symmetric spectrum
$X$ is defined as the colimit
   $$\hat\pi_k(X)=\colim_n\pi_{k+n}X_n.$$
Denote by $\gamma X$ a stably fibrant model of $X$ in $Sp^\Sigma$.
The $k$-th true homotopy group of $X$ is given by
   $$\pi_kX=\hat\pi_k(\gamma X),$$
the naive homotopy groups of the symmetric spectrum $\gamma X$.

Naive and true homotopy groups of $X$ can considerably be different
in general (see, e.g.,~\cite{HSS,Sch}). The true homotopy groups
detect stable equivalences, and are thus more important than the
naive homotopy groups. There is an important class of {\it
semistable\/} symmetric spectra within which
$\hat\pi_*$-isomorphisms coincide with $\pi_*$-isomorphisms. Recall
that a symmetric spectrum is semistable if some (hence any) stably
fibrant replacement is a $\pi_*$-isomorphism. Suspension spectra,
Eilenberg--Mac Lane spectra, $\Omega$-spectra or $\Omega$-spectra
from some point $X_n$ on are examples of semistable symmetric
spectra (see~\cite{Sch}).
Semistability
is preserved under suspension, loop, wedges and shift.

A symmetric spectrum $X$ is {\it $n$-connected\/} if the true
homotopy groups of $X$ are trivial for $k\geq n$. The spectrum $X$
is {\it connective\/} is it is $(-1)$-connected, i.e., its true
homotopy groups vanish in negative dimensions. $X$ is {\it bounded
below\/} if $\pi_i(X)=0$ for $i\ll 0$.

\begin{defs}\label{basic}{\rm
(1) Following~\cite{SS} a {\it spectral category\/} is a category
$\cc O$ which is enriched over the category $Sp^\Sigma$ of symmetric
spectra (with respect to smash product, i.e., the monoidal closed
structure of \cite[2.2.10]{HSS}). In other words, for every pair of
objects $o,o'\in\cc O$ there is a morphism symmetric spectrum $\cc
O(o,o')$, for every object $o$ of $\cc O$ there is a map from the
sphere spectrum $S$ to $\cc O(o,o)$ (the ``identity element" of
$o$), and for each triple of objects there is an associative and
unital composition map of symmetric spectra $\cc O(o',o'')\wedge\cc
O(o,o') \to\cc O(o,o'')$. An $\cc O$-module $M$ is a contravariant
spectral functor to the category $Sp^\Sigma$ of symmetric spectra,
i.e., a symmetric spectrum $M(o)$ for each object of $\cc O$
together with coherently associative and unital maps of symmetric
spectra $M(o)\wedge\cc O(o',o)\to M(o')$ for pairs of objects
$o,o'\in\cc O$. A morphism of $\cc O$-modules $M\to N$ consists of
maps of symmetric spectra $M(o)\to N(o)$ strictly compatible with
the action of $\cc O$. The category of $\cc O$-modules will be
denoted by $\Mod\cc O$.

(2) A {\it spectral functor\/} or a {\it spectral homomorphism\/}
$F$ from a spectral category $\cc O$ to a spectral category $\cc O'$
is an assignment from $\Ob\cc O$ to $\Ob\cc O'$ together with
morphisms $\cc O(a,b)\to\cc O'(F(a),F(b))$ in $Sp^\Sigma$ which
preserve composition and identities.

(3) The {\it monoidal product\/} $\cc O\wedge\cc O'$ of two spectral
categories $\cc O$ and $\cc O'$ is the spectral category where
$\Ob(\cc O\wedge\cc O'):=\Ob\cc O\times\Ob\cc O'$ and $\cc
O\wedge\cc O'((a,x),(b,y)):= \cc O(a,b)\wedge\cc O'(x,y)$.

(3') A monoidal spectral category consists of a spectral category $\cc O$ equipped with
a spectral functor
$\diamond: \cc O \wedge \cc  O \to \cc O$,
a unit $u \in Ob \cc O$, a $Sp^{\Sigma}$-natural associativity isomorphism
and two $Sp^{\Sigma}$-natural unit isomorphisms. Symmetric monoidal spectral
categories are defined similarly.


(4) A spectral category $\cc O$ is said to be {\it connective\/} if
for any objects $a,b$ of $\cc O$ the spectrum $\cc O(a,b)$ is
connective.

(5) By a ringoid over $Sm/k$ we mean a preadditive category $\cc R$
whose objects are those of $Sm/k$ together with a functor
   $$\rho:Sm/k\to\cc R,$$
which is identity on objects. Every such ringoid gives rise to a
spectral category $\cc O_{\cc R}$ whose objects are those of $Sm/k$
and the morphisms spectrum $\cc O_{\cc R}(X,Y)$, $X,Y\in Sm/k$, is
the Eilenberg--Mac~Lane spectrum $H\cc R(X,Y)$ associated with the
abelian group $\cc R(X,Y)$. Given a map of schemes $\alpha$, its
image $\rho(\alpha)$ will also be denoted by $\alpha$, dropping
$\rho$ from notation.

(6) By a spectral category over $Sm/k$ we mean a spectral category
$\cc O$ whose objects are those of $Sm/k$ together with a spectral
functor
   $$\sigma:\cc O_{naive}\to\cc O,$$
which is identity on objects. Here $\cc O_{naive}$ stands for the
spectral category whose morphism spectra are defined as
   $$\cc O_{naive}(X,Y)_p=\Hom_{Sm/k}(X,Y)_+\wedge S^p$$
for all $p\geq 0$ and $X,Y\in Sm/k$.

It is straightforward to verify that the category of $\cc
O_{naive}$-modules can be regarded as the category of presheaves
$Pre^\Sigma(Sm/k)$ of symmetric spectra on $Sm/k$. This is used in
the sequel without further comment.

}\end{defs}

Let $\cc O$ be a spectral category and let $\Mod\cc O$ be the
category of $\cc O$-modules. Recall that the projective stable model
structure on $\Mod\cc O$ is defined as follows (see~\cite{SS}). The
weak equivalences are the objectwise stable weak equivalences and
fibrations are the objectwise stable projective fibrations. The
stable projective cofibrations are defined by the left lifting
property with respect to all stable projective acyclic fibrations.

Recall that the Nisnevich topology is generated by elementary
distinguished squares, i.e. pullback squares
   \begin{equation}\label{squareQ}
    \xymatrix{\ar@{}[dr] |{\textrm{$Q$}}U'\ar[r]\ar[d]&X'\ar[d]^\phi\\
              U\ar[r]_\psi&X}
   \end{equation}
where $\phi$ is etale, $\psi$ is an open embedding and
$\phi^{-1}(X\setminus U)\to(X\setminus U)$ is an isomorphism of
schemes (with the reduced structure). Let $\cc Q$ denote the set of
elementary distinguished squares in $Sm/k$ and let $\cc O$ be a
spectral category over $Sm/k$. By $\cc Q_{\cc O}$ denote the set of
squares
   \begin{equation}\label{squareOQ}
    \xymatrix{\ar@{}[dr] |{\textrm{$\cc O Q$}}\cc O(-,U')\ar[r]\ar[d]&\cc O(-,X')\ar[d]^\phi\\
              \cc O(-,U)\ar[r]_\psi&\cc O(-,X)}
   \end{equation}
which are obtained from the squares in $\cc Q$ by taking $X\in Sm/k$
to $\cc O(-,X)$. The arrow $\cc O(-,U')\to\cc O(-,X')$ can be
factored as a cofibration $\cc O(-,U')\rightarrowtail Cyl$ followed
by a simplicial homotopy equivalence $Cyl\to\cc O(-,X')$. There is a
canonical morphism $A_{\cc O Q}:=\cc O(-,U)\bigsqcup_{\cc O(-,U')}
Cyl\to\cc O(-,X)$.

\begin{defs}[see~\cite{GP}]{\rm
I. The {\it Nisnevich local model structure\/} on $\Mod\cc O$ is the
Bousfield localization of the stable projective model structure with
respect to the family of projective cofibrations
   \begin{equation*}\label{no}
    \cc N_{\cc O}=\{\cyl(A_{\cc O Q}\to\cc O(-,X))\}_{\cc Q_{\cc O}}.
   \end{equation*}
The homotopy category for the Nisnevich local model structure will
be denoted by $\shnis\cc O$. In particular, if $\cc O=\cc O_{naive}$
then we have the Nisnevich local model structure on
$Pre^\Sigma(Sm/k)=\Mod\cc O_{naive}$ and we shall write $\shnis(k)$
to denote $\shnis\cc O_{naive}$.

II. The {\it motivic model structure\/} on $\Mod\cc O$ is the
Bousfield localization of the Nisnevich local model structure with
respect to the family of projective cofibrations
   \begin{equation*}\label{ao}
    \cc A_{\cc O}=\{\cyl(\cc O(-,X\times\bb A^1)\to\cc O(-,X))\}_{X\in Sm/k}.
   \end{equation*}
The homotopy category for the motivic model structure will be
denoted by $\sheff\cc O$. In particular, if $\cc O=\cc O_{naive}$
then we have the motivic model structure on
$Pre^\Sigma(Sm/k)=\Mod\cc O_{naive}$ and we shall write write
$\sheff(k)$ to denote $\sheff\cc O_{naive}$.

}\end{defs}

\begin{defs}[see~\cite{GP}]\label{Nis_and_Mot_exc}{\rm
I. We say that $\cc O$ is {\it Nisnevich excisive\/} if for every
elementary distinguished square $Q$
    \begin{equation*}
    \xymatrix{\ar@{}[dr] |{\textrm{$Q$}}U'\ar[r]\ar[d]&X'\ar[d]^\phi\\
              U\ar[r]_\psi&X}
   \end{equation*}
the square $\cc O Q$~\eqref{squareOQ} is homotopy pushout in the
Nisnevich local model structure on $Pre^\Sigma(Sm/k)$.

II. $\cc O$ is {\it motivically excisive\/} if:

\begin{itemize}
\item[(A)] for every elementary distinguished square $Q$ the square $\cc O
Q$~\eqref{squareOQ} is homotopy pushout in the motivic model
structure on $Pre^\Sigma(Sm/k)$ and

\item[(B)] for every $X\in Sm/k$ the natural map
   $$\cc O(-,X\times\bb A^1)\to\cc O(-,X)$$
is a weak equivalence in the motivic model structure on
$Pre^\Sigma(Sm/k)$.
\end{itemize}

}\end{defs}

Recall that a sheaf $\cc F$ of abelian groups in the Nisnevich
topology on $Sm/k$ is {\it strictly $\bb A^1$-invariant\/} if for
any $X\in Sm/k$, the canonical morphism
   $$H^*_{\nis}(X,\cc F)\to H^*_{\nis}(X\times\bb A^1,\cc F)$$
is an isomorphism.

\begin{defs}\label{vsp}{\rm
Let $(\cc O, \diamond, pt)$ be a {\it symmetric monoidal} spectral category over $Sm/k$ together with the
structure spectral functor $\sigma:\cc O_{naive}\to\cc O$
and an additive functor
$\bb ZF_*(k)\xrightarrow{\epsilon}\pi_0\cc O$. We say
that $((\cc O,\diamond,pt),\sigma,\epsilon)$ is a {\it symmetric monoidal $V$-spectral category\/} if

\begin{enumerate}
\item $\cc O$ is connective and Nisnevich excisive;
\item the structure map $\rho: Sm/k \to \pi_0\cc O$ induced by $\sigma$ equals $\epsilon \circ in$, 
where $in: Sm/k \to \bb ZF_*(k)$ is the graphic functor.
\end{enumerate}
}
\end{defs}

\begin{rem}\label{additivity} {\rm
Since $\cc O$ is connective and Nisnevich excisive, for each
$\cc O$-module $M$ and each integer $i$ the presheaf
$\pi_i(M)|_{Sm/k}$ is {\it radditive} (the restriction is taken via the $\rho$).
That is $\pi_i(M)(\emptyset)=0$ and $\pi_i(M)(X_1\sqcup X_2)=\pi_i(M)(X_1)\times \pi_i(M)(X_2)$.
Particularly, the functor $\pi_i(M)|_{\bb ZF_*(k)}$ is additive. So, $\pi_i(M)|_{\bb ZF_*(k)}$ is
a {\it presheaf of Abelian groups on} $\bb ZF_*(k)$
in the sense of \cite[Def. 2.13]{GP3}
(the restriction is taken via the $\epsilon$).
}
\end{rem}

We note that if $(\cc O, \diamond, pt)$ is a symmetric monoidal spectral category over $Sm/k$, 
then for every
$\cc O$-module $M$ and any smooth scheme $U$, the presheaf of
symmetric spectra
   $$\underline{\Hom}(U,M):=M(-\times U)$$
is an $\cc O$-module. Moreover, $M(-\times U)$ is functorial in $U$.

\begin{lem}\label{pepe}
Every symmetric monoidal $V$-spectral category $\cc O$ is motivically excisive.
\end{lem}

\begin{proof}
Every symmetric monoidal $V$-spectral category is, by definition, Nisnevich excisive.
Since there is an action of smooth schemes on $\cc O$, the
fact that $\cc O$ is motivically excisive is proved similar
to~\cite[5.8]{GP}. 
\end{proof}

\begin{defs}\label{SHnis}{\rm
Let $((\cc O,\diamond,pt),\sigma,\epsilon)$ be a symmetric monoidal $V$-spectral category. Since it is both Nisnevich
and motivically excisive, it follows from~\cite[5.13]{GP} that the
pair of natural adjoint fuctors
   $$\xymatrix{{\Psi_*}:Pre^\Sigma(Sm/k)\ar@<0.5ex>[r]&\Mod\cc O:{\Psi^*}\ar@<0.5ex>[l]}$$
induces a Quillen pair for the Nisnevich local projective
(respectively motivic) model structures on $Pre^\Sigma(Sm/k)$ and
$\Mod\cc O$. In particular, one has adjoint functors between
triangulated categories
   \begin{equation}\label{adjoint}
    {\Psi_*}: \text{SH}^{nis}(\cc O_{naive})\rightleftarrows \text{SH}^{nis}(\cc O):{\Psi^*}\quad\textrm{ and }
    \quad {\Psi_*}:\text{SH}^{mot}(\cc O_{naive})\rightleftarrows \text{SH}^{mot}(\cc O):{\Psi^*}.
   \end{equation}
}   
\end{defs}   

\section{The triangulated category $D\cc O_-^{eff}(k)$}\label{dominus}

In this section we work with a symmetric monoidal $V$-spectral category 
$((\cc O,\diamond,pt),\sigma,\epsilon)$ in the sense of Definition \ref{vsp}.
We work in this section with the category $\text{SH}^{nis}(\cc O)$ as in Definition \ref{SHnis}.

Let $M$ be an $\cc O$-module. By Remark \ref{additivity} its $\pi_0\cc O$-presheaves $\pi_i(M)$
restricted via the $\epsilon$ to the additive category $\bb ZF_*(k)$
are $\bb ZF_*(k)$-{\it presheaves of Abelian groups} in the sense of \cite[Def. 2.13]{GP3}. 
Thus, by \cite[Lemma 4.5]{Voe2} and \cite[Cor. 2.17]{GP3} the associated Nisnevich sheaf
$\pi^{nis}_i(M)$ is canonically a $\bb ZF_*(k)$-presheaves of Abelian groups 
(possibly it is not a $\pi_0\cc O$-presheaf).

We shall often work with simplicial $\cc O$-modules
$M[\bullet]$. The {\it realization\/} of $M[\bullet]$ is the $\cc
O$-module $|M|$ defined as the coend
   $$|M|=\Delta[\bullet]_+\wedge_{\Delta} M[\bullet]$$
of the functor $\Delta[\bullet]_+\wedge
M[\bullet]:\Delta\times\Delta^{\op}\to\Mod\cc O$. Here $\Delta[n]$
is the standard simplicial $n$-simplex.

Recall that the simplicial ring $k[\Delta]$ is defined as
   $$k[\Delta]_n=k[x_0,\ldots,x_n]/(x_0+\cdots+x_n-1).$$
By $\Delta^{\cdot}$ we denote the cosimplicial affine scheme
$\spec(k[\Delta])$. 
Given an $\cc O$-module $M$, we set
   $$C_*(M):=|\underline{\Hom}(\Delta^{\cdot},M)|.$$
Note that $C_*(M)$ is an $\cc O$-module and is functorial in $M$. 
{\bf Our} $C_*(M)$ {\bf is different of} $C_*(M)$ {\bf used in} \cite[Sect. 3]{GP2}.

\begin{defs}[Definition 3.3 in \cite{GP2}]
\label{boundedOmod}{\rm
The $\cc O$-motive $M_{\cc O}(X)$ of a smooth algebraic variety
$X\in Sm/k$ is the $\cc O$-module $C_*(\cc O(-,X))$. We say that an
$\cc O$-module $M$ is {\it bounded below\/} if for $i\ll 0$ the
Nisnevich sheaf $\pi_i^{\nis}(M)$ is zero. $M$ is {\it
$n$-connected\/} if $\pi_i^{\nis}(M)$ are trivial for $i\leq n$. $M$
is {\it connective\/} is it is $(-1)$-connected, i.e.,
$\pi_i^{\nis}(M)$ vanish in negative dimensions.
}
\end{defs}

\begin{defs}[\cite{GP2}]\label{DOminus}{\rm
Denote by $\Mod_{-}\cc O$ the full subcategory of
of bounded below $\cc O$-modules. \\
Denote by $D\cc O_-(k)$ the full triangulated subcategory of
$SH^{nis}(\cc O)$ of bounded below $\cc O$-modules. We also denote by
$D\cc O_-^{eff}(k)$ the full triangulated subcategory of $D\cc
O_-(k)$ of those $\cc O$-modules $M$ such that each
$\pi_0\cc O$-presheaf
$\pi_i(M)$ regarded via the functor $\epsilon$ as a $\bb ZF_*(k)$-presheaf of Abelian groups
is {\it homotopy invariant and stable} in the sense of
\cite[Def. 2.13, 2.14]{GP3}. 

The category $D\cc O_-^{eff}(k)$ is an
analog of Voevodsky's triangulated category
$DM_-^{eff}(k)$.
}
\end{defs}

\begin{lem}[Corollary 3.4 in \cite{GP2}]
\label{porto}{\rm
If an $\cc O$-module $M$ is bounded below (respectively
$n$-connected) then so is $C_*(M)$. In particular, the
$\cc O$-motive $M_{\cc O}(X)$ of any smooth algebraic variety $X\in Sm/k$
is connective.
}
\end{lem}

\begin{rem}\label{1st_endo_funct}{\rm
By Lemma \ref{porto} the assignment $M\mapsto C_*(M)$ is a functor $C_*: \Mod_{-}\cc O\to \Mod_{-}\cc O$.
}
\end{rem}

\begin{lem}[Compare with Lemma 3.5 in \cite{GP2}]
\label{spain}{\rm
The functor $C_*: \Mod_{-}\cc O\to \Mod_{-}\cc O$ respects local equivalences and 
induces a triangulated endofunctor
   $$C_*: D\cc O_-(k)\to D\cc O_-(k)$$
}   
\end{lem}

\begin{thm}[Compare with Theorem 3.5 in \cite{GP2}]
\label{neploho}{\rm
Let $(\cc O,\diamond, pt)$ be a symmetric monoidal $V$-spectral category. 
Consider the full triangulated subcategory $\cc T$ of $SH^{nis}(\cc O)$ generated by the compact objects
$\cone(\cc O(-,X\times\bb A^1)\to\cc O(-,X)),\ X\in Sm/k.$
Then the triangulated endofunctor
   $$C_*:D\cc O_-(k)\to D\cc O_-(k)$$
as in Lemma \ref{spain} lands in $D\cc O_-^{eff}(k)$. The kernel of $C_*$ is $\cc T_-:=\cc
T\cap D\cc O_-(k)$. Moreover, $C_*$ is left adjoint to the inclusion
functor
   $$i:D\cc O_-^{eff}(k)\to D\cc O_-(k)$$
and $D\cc O_-^{eff}(k)$ is triangle equivalent to the quotient
category $D\cc O_-(k)/\cc T_-$ \ .
}
\end{thm}

\section{The main symmetric monoidal strict $V$-spectral category}\label{The_Category}
We construct in this section our
main symmetric monoidal strict $V$-spectral category $(\bb {F}\text{r},\diamond, pt)$.

First construct a spectral category $\bb {F}\text{r}$.
Its objects are those of $Sm/k$.
To each pair $Y,X\in Sm/k$ we assign a symmetric spectrum
$\bb {F}\text{r}(Y,X)$. The latter is described as follows. Its terms are the functors
$A \mapsto \bb {F}\text{r}(Y,X)_A=Fr_A(Y,X\otimes S^A)$ (here $A$ runs over the category of finite sets
and their isomorphisms). The structure maps are defined by the obvious compositions
$$\epsilon_{A,B}: Fr_A(Y,X\otimes S^A)\wedge S^B\to Fr_{A}(Y,X\otimes S^{A\sqcup B})\hookrightarrow Fr_{A\sqcup B}(Y,X\otimes S^{A\sqcup B}).$$
For each triple $Z,Y,X\in Sm/k$ there is an obvious symmetric spectra morphism
$$\circ_{Z,Y,X}: \bb {F}\text{r}(Y,X)\wedge \bb {F}\text{r}(Z,Y)\to \bb {F}\text{r}(Z,X) \ \ \text{(the composition law)}.$$
It is uniquely determined by simplicial set morphisms
$\bb {F}\text{r}(Y,X)_A\wedge \bb {F}\text{r}(Z,Y)_B\to \bb {F}\text{r}(Z,X)_{A\sqcup B}$
which on $n$-simplices are given by the set maps
$$Fr_A(Y,X\otimes (S^A)_n)\wedge Fr_B(Z,Y\otimes (S^B)_n)\to Fr_A(Y\otimes (S^B)_n,X\otimes (S^A)_n\otimes (S^B)_n)\wedge Fr_B(Z,Y\otimes (S^B)_n)\to$$
$$\to Fr_{A\sqcup B}(Z,X\otimes (S^{A\sqcup B})_n).$$
In details, the set map is given by
$$(\alpha,\beta)\mapsto (\alpha\otimes id_{(S^B)_n},\beta)\mapsto (\alpha\otimes id_{(S^B)_n})\circ \beta.$$
For each $X\in Sm/k$ the identity morphism $id_X$ gives rise to the symmetric spectra morphism
$u_X: \bb S\to \bb {F}\text{r}(X,X)$.
We formed a spectral category $\bb {F}\text{r}$ and
a spectral functor
$\sigma: \cc O_{naive}\to \bb {F}\text{r}$,
which is identity on objects. The pair
$(\bb {F}\text{r},\sigma)$
is a spectral category over $Sm/k$
in the sense of Definition \ref{basic}(6).

Equip now the spectral category $\bb {F}\text{r}$ with a spectral functor
$\diamond: \bb {F}\text{r}\wedge \bb {F}\text{r}\to \bb {F}\text{r}$
(taking $(X_1,X_2)$ to $X_1\times X_2$),
a unit $u\in \bb {F}\text{r}$,
a $Sp^{\Sigma}$-natural associativity isomorphism $a$
and two $Sp^{\Sigma}$-natural unit isomorphisms $u_l$, $u_r$
and a twist isomorphism
$tw: \bb {F}\text{r}\wedge \bb {F}\text{r}\to \bb {F}\text{r}\wedge \bb {F}\text{r}$
and a spectral functor isomorphism
$\Phi: \diamond\to \diamond\circ tw$
such that the data
$$(\bb {F}\text{r}, \diamond, tw, \Phi, u, a, u_l, u_r)$$
form a symmetric monoidal spectral category.

First construct the spectral functor
$\diamond$. On objects it takes an object $(X_1,X_2)\in Sm/k\times Sm/k$ to $X_1\times X_2 \in Sm/k$.
To construct $\diamond$ on morphisms it sufficient to construct certain symmetric spectra morphisms
$$\diamond_{(V,Y),(U,X)}: \bb {F}\text{r}(V,U)\wedge \bb {F}\text{r}(Y,X)\xrightarrow{} \bb {F}\text{r}(V\times Y,U\times X)$$
and check that they satisfy the expected properties. To construct the morphism $\diamond_{(V,Y),(U,X)}$ it is sufficient to construct
simplicial set morphisms
$$\boxtimes_{(V,Y),(U,X),A,B}: \bb {F}\text{r}(V,U)_A\wedge \bb {F}\text{r}(Y,X)_B \xrightarrow{} \bb {F}\text{r}(V\times Y,U\times X)_{A\sqcup B}$$
subjecting the known properties.
The latter are given on $n$-simplices by the exterior product maps
$$\boxtimes_{(V,Y),(U,X),A,B,\ n}: Fr_A(V,U\otimes (S^A)_n)\wedge Fr_B(Y,X\otimes (S^B)_n) \to Fr_{A\sqcup B}(V\times Y,(U\times X) \otimes (S^{A\sqcup B})_n).$$
We constructed the spectral functor
$\diamond$.

Second we take the point $pt:=Spec (k)$ as the unit of the spectral category $\bb {F}\text{r}$
and we skip constructions of desired $a$, $u_l$, $u_r$ (they are obvious).

Third we construct the twist spectral categories isomorphism
$tw: \bb {F}\text{r}\wedge \bb {F}\text{r}\to \bb {F}\text{r}\wedge \bb {F}\text{r}$.
On objects it takes $(X_1,X_2)$ to $(X_2,X_1)$. On morphisms it is determined by certain symmetric spectra isomorphisms
$$tw_{(V,Y),(U,X)}: \bb {F}\text{r}(V,U)\wedge \bb {F}\text{r}(Y,X)\xrightarrow{} \bb {F}\text{r}(Y,X)\wedge \bb {F}\text{r}(V,U).$$
In turn the $tw_{(V,Y),(U,X)}$ is determined by the family of simplicial set isomorphisms (switching factors)
$$tw^C_{A,B}: \bb {F}\text{r}(V,U)_A\wedge \bb {F}\text{r}(Y,X)_B \to \bb {F}\text{r}(Y,X)_B\wedge \bb {F}\text{r}(V,U)_A.$$
Here for each finite set $C$ the ordered pairs $(A,B)$ run over all subsets $A\subseteq C$, $B\subseteq C$ such that
$A\cup B=C$ and $A\cap B=\emptyset$.

Finally we construct the desired spectral functor isomorphism
$\Phi: \diamond\to \diamond\circ tw$. It is the assignment
$(V,Y)\mapsto \Phi(V,Y)=[\tau_{V,Y}: V\times Y \to Y\times X]$. Here the switching factors
isomorphism $\tau_{V,Y}$ is regarded as a point in $Fr_0(V\times Y, Y\times V)$.
So, it is regarded as a symmetric spectra morphism
$\bb S\xrightarrow{\Phi(V,Y)} \bb {F}\text{r}(V\times Y, Y\times V)$.
It's easy to check that $\Phi$ is a spectral functor isomorphism indeed.

We left to the reader to check that the data
$(\bb {F}\text{r}, \diamond, tw, \Phi, u, a, u_l, u_r)$
form a symmetric monoidal spectral category.

\section{Properties of the main spectral category}
Let $((\bb {F}\text{r},\diamond,pt), \sigma:\cc O_{naive}\to \bb {F}\text{r})$ be the {\it symmetric monoidal} spectral category over $Sm/k$
as in Section \ref{The_Category}. 
\begin{lem}\label{epsilon}{\rm
There is an additive functor $\bb ZF_*(k)\xrightarrow{\epsilon}\pi_0(\bb {F}\text{r})$ such that 
the data $((\cc O, \diamond, pt),\sigma,\epsilon)$
is a symmetric monoidal $V$-spectral category in the sense of Definition \ref{vsp}.
}
\end{lem}
Applying now Lemma \ref{pepe} we get the following 
\begin{cor}\label{Nis_and_Mot_exc_true}{\rm
The symmetric monoidal spectral category $(\bb {F}\text{r}, \diamond, pt, tw, \Phi, u, a, u_l, u_r)$
as in Section \ref{The_Category} is Nisnevich and Motivically excisive in the sense of
\cite{GP}
(see Definition \ref{Nis_and_Mot_exc}).
}
\end{cor}

The following definition is just Definition \ref{boundedOmod} adapted to the category $\Mod \bb {F}\text{r}$.
\begin{defs}\label{boundedFrmod}{\rm
The $\bb {F}\text{r}$-motive $M_{\bb {F}\text{r}}(X)$ of a smooth algebraic variety
$X\in Sm/k$ is the $\bb {F}\text{r}$-module $C_*(\bb {F}\text{r}(-,X))$. We say that an
$\bb {F}\text{r}$-module $M$ is {\it bounded below\/} if for $i\ll 0$ the
Nisnevich sheaf $\pi_i^{\nis}(M)$ is zero. $M$ is {\it
$n$-connected\/} if $\pi_i^{\nis}(M)$ are trivial for $i\leq n$. $M$
is {\it connective\/} is it is $(-1)$-connected, i.e.,
$\pi_i^{\nis}(M)$ vanish in negative dimensions.

}\end{defs}

\begin{defs}{\rm
Denote by $\Mod \bb {F}\text{r}_{-}$ the full subcategory of
of bounded below $\bb {F}\text{r}$-modules. \\
Denote by $\text{D}\bb {F}\text{r}_-(k)$ the full triangulated subcategory of
$SH^{nis}(\bb {F}\text{r})$ of bounded below $\bb {F}\text{r}$-modules. We also denote by
$\text{D}\bb {F}\text{r}_-^{eff}(k)$
the full triangulated subcategory of
$\text{D}\bb {F}\text{r}_-(k)$
of those $\bb {F}\text{r}$-modules $M$ such that each 
$\bb ZF_*(k)$-presheaf
$\pi_i(M)|_{\bb ZF_*(k)}$ is 
{\it homotopy invariant and stable} 
in the sense of 
\cite[Def. 2.13, 2.14]{GP3}. 
}
\end{defs}
In certain sense
$\text{D}\bb {F}\text{r}_-^{eff}(k)$ is an
analog of Voevodsky's triangulated category
$DM_-^{eff}(k)$ 
\cite{Voe1}.

\begin{defs}\label{prekrasno}{\rm
The triangulated category $\text{D}\bb {F}\text{r}_-^{eff}(k)$
is called
{\it the triangulated category of effective $\bb {F}\text{r}$-motives.
}
}
\end{defs}

One can prove the following 
\begin{thm}\label{DFr_eff_and_SH_eff}{\rm
There is a natural triangulated equivalence between the triangulated categories
$\text{D}\bb {F}\text{r}_-^{eff}(k)$
and the Voevodsky category
$SH^{eff}_-(k)$.
}
\end{thm}
A sketch of a proof of this result will be presented in the next section.

\section{Triangulated equivalences $SH^{eff}(k)\rightleftarrows \text{D}\bb {F}\text{r}_-^{eff}(k)$ }
We construct in this section triangulated equivalences (quasi-inverse to each the other)
$$\bb M^{\bb {F}\text{r}}_{\text{eff}}: SH^{eff}_-(k)\rightleftarrows \text{D}\bb {F}\text{r}_-^{eff}(k): \bb M_{\text{SH}}^{\text{eff}}.$$
To construct these functors we need preliminaries.
Let $\bb G^{\wedge 1}_m\in \Delta^{op}(Fr_0(k))$ be as in \cite[Notation 8.1]{GP4}. 
Let $\bb G_m^{\wedge n}$ be the $n$th monoidal power $\bb G_m^{\wedge 1}$ be as in \cite[Notation 8.1]{GP4}.
The category $Pre^{\Sigma}_{S^1, \bb G^{\wedge 1}_m}(Sm/k)$ of presheaves of symmetric bispectra
can be regarded as the category of
symmetric $\bb G^{\wedge 1}_m$-spectra in the category
$\Mod\cc O_{naive}$
of presheaves of symmetric spectra
(see Definition \ref{basic}).

Similarly we can (and will) consider a category of symmetric $\bb G^{\wedge 1}_m$-spectra in the category
$\Mod\bb {F}\text{r}$.
It follows from~\cite[5.13]{GP} that there is 
a pair of natural adjoint fuctors
   $$\xymatrix{{\Phi_*}: Pre^{\Sigma}_{S^1, \bb G^{\wedge 1}_m}(Sm/k)=Sp_{\bb G^{\wedge 1}_m}(\Mod\cc O_{naive})\ar@<0.5ex>[r]& Sp_{\bb G^{\wedge 1}_m}(\Mod\bb {F}\text{r}):{\Phi^*}\ar@<0.5ex>[l]}$$
There is another pair of adjoint functors
$$\xymatrix{{\Sigma^{\infty}_{\bb {F}\text{r}(\bb G^{\wedge 1}_m)}}: \Mod\bb {F}\text{r} \ar@<0.5ex>[r]& Sp_{\bb G^{\wedge 1}_m}(\Mod\bb {F}\text{r}): 
\Omega^{\infty}_{\bb {F}\text{r}(\bb G^{\wedge 1}_m)}\ar@<0.5ex>[l]}$$
Here $\bb {F}\text{r}(\bb G^{\wedge 1}_m)$ stands for the $\bb {F}\text{r}$-module represented by the simplicial scheme $\bb G^{\wedge 1}_m$.

For each $\bb {F}\text{r}$-module $M$ consider the $\bb {F}\text{r}$-module
$C_*(M):=|\underline{\Hom}(\Delta^{\cdot},M)|$ as in Section \ref{dominus}. 
By Lemma \ref{spain} and Theorem \ref{neploho}
the endo-functor $C_*: \Mod\bb {F}\text{r}_{-}\to \Mod\bb {F}\text{r}_{-}$ 
induces a trangulated functor
$C_*: \text{D}\bb {F}\text{r}_-(k) \to \text{D}\bb {F}\text{r}_-^{eff}(k)$. 
By Theorem \ref{neploho}
the pair of triangulated functors 
\begin{equation}\label{Sigma_Omega_2}
    C_*: \text{D}\bb {F}\text{r}_-(k) \rightleftarrows \text{D}\bb {F}\text{r}_-^{eff}(k):i
   \end{equation}
is a pair of adjoint triangulated functors (here $i$ is the inclusion functor).\\
Let $\bb {F}\text{r}(n)=\text{M}_{\bb {F}\text{r}}(\bb G^{\wedge n}_m)$ be the $\bb {F}\text{r}$-motive of $\bb G^{\wedge n}_m$. 
For each cofibrant object $E$ in the projective model structure on $\Mod \bb {F}\text{r}$ put
$E(n)=E\otimes^{\bb {F}\text{r}}\bb {F}\text{r}(n)$. It is a cofibrant object 
in the projective model structure on $\Mod \bb {F}\text{r}$. Clearly, 
$\Sigma^{\infty}_{\bb {F}\text{r}(1)}(E):=(E, E(1), E(2), ... )$ is naturally an object of $Sp_{\bb G^{\wedge 1}_m}(\Mod\bb {F}\text{r}$).
\begin{defs}{\rm
Let $E\mapsto E^c$ be the cofibrant replacement in the projective model structure on $Pre^{\Sigma}_{S^1, \bb G^{\wedge 1}_m}(Sm/k)$. Put\\
$\bb M^{\bb {F}\text{r}}(E)=(C_* \circ \Omega^{\infty}_{\bb {F}\text{r}(1)} \circ \Phi_*)(E^c)=
\Omega^{\infty}_{\bb G^{\wedge 1}_m}C_*\bb F\text{r}(E^c) \in \Mod \bb {F}\text{r}$. \\
Let $\cc E\mapsto \cc E^c$ be the cofibrant replacement in the projective model structure on $\Mod \bb {F}\text{r}$. Put\\
$\bb M_{\text{SH}}(\cc E)=\Phi^*(\Sigma^{\infty}_{\bb {F}\text{r}(1)}(\cc E^c))\in Pre^{\Sigma}_{S^1, \bb G^{\wedge 1}_m}(Sm/k)$.
Thus,\\
$\bb M_{\text{SH}}(\cc E)=\text{the object} \ \Sigma^{\infty}_{\bb {F}\text{r}(1)}(\cc E^c) \ \text{of} \ Sp_{\bb G^{\wedge 1}_m}(\Mod\bb {F}\text{r}) \ 
\text{regarded as an object in} \ Pre^{\Sigma}_{S^1, \bb G^{\wedge 1}_m}(Sm/k)$.
}
\end{defs}
A proof of the following result is postponed to the next preprint. 
It can be given in the spirit of the proofs as in \cite[Section 2]{GP5}.
\begin{thm}\label{VeryMain}{\rm
$\bullet$ The functor $\bb M_{\text{SH}}$ induces a triangulated equivalence \\
$\bb M^{eff}_{\text{SH}}: \text{D}\bb {F}\text{r}_-^{eff}(k)\to SH^{eff}_-(k)$ \\
between these triangulated categories; \\
$\bullet$ A triangulated functor $\bb M^{\bb {F}\text{r}}_{eff} : SH^{eff}_-(k)\to \text{D}\bb {F}\text{r}_-^{eff}(k)$ \\ 
quasi-inverse to 
$\bb M^{eff}_{\text{SH}}$
is induced by the functor \\
$\bb M^{\bb {F}\text{r}}: Pre^{\Sigma}_{S^1, \bb G^{\wedge 1}_m}(Sm/k)\to \Mod \bb {F}\text{r}$.
}
\end{thm}


\begin{thebibliography}{99}
\bibitem{AGP} A. Ananyevskiy, G.~Garkusha, I.~Panin, Cancellation theorem for
             framed motives of algebraic varieties, Advances in Mathematics, 383 (2021) 107681,
             preprint arXiv:1601.06642.
\bibitem{DP} A. Druzhinin, I. Panin, Surjectivity of the \'{e}tale excision map for homotopy invariant framed presheaves,\\
Trudy MIAN, 2022, to appear.
\bibitem{GP} G. Garkusha, I. Panin, K-motives of algebraic varieties, Homology Homotopy Appl. 14(2) (2012), 211-264.
\bibitem{GP2} G. Garkusha, I. Panin, The triangulated category of K-motives $DK^{eff}_{minus}(k)$,
Journal of K-Theory , Volume 14 , Issue 1 , August 2014 , pp. 103--137.
\bibitem{GP1} G. Garkusha, I. Panin, On the motivic spectral sequence, J. Inst. Math. Jussieu 17 (2018), no. 1, 137--170.
\bibitem{GP3} G. Garkusha and I. Panin, Homotopy invariant presheaves with framed transfers, Cambridge
J. Math. 8 (2020), no. 1, 1--94.
\bibitem{GP4} G. Garkusha and I. Panin, FRAMED MOTIVES OF ALGEBRAIC VARIETIES (AFTER V. VOEVODSKY),
J. Amer. Math. Soc. 34 (2021), 261--313.
\bibitem{GP5} G. Garkusha and I. Panin, Triangulated categories of framed bispectra and framed motives,
Algebra i Analiz, Vol.34 (2022), No.6, 135--169.
\bibitem{GNP} G. Garkusha, A. Neshitov, and I. Panin, Framed motives of relative motivic spheres,
TRANSACTIONS OF THE AMERICAN MATHEMATICAL SOCIETY, 374, No.7, 2021, 5131--5161.
\bibitem{GH} T. Geisser, L. Hesselholt, Topological cyclic homology of
         schemes, in Algebraic K-theory (Seattle, WA, 1997), Proc. Sympos.
         Pure Math. 67, Amer. Math. Soc., Providence, RI, 1999, pp. 41-87.
\bibitem{Gr} D. Grayson, Weight filtrations via commuting automorphisms, K-Theory 9 (1995), 139-172.
\bibitem{HSS} M. Hovey, B. Shipley, J. Smith, Symmetric spectra, J. Amer. Math. Soc. 13 (2000), 149-208.
\bibitem{Jar1} J. F. Jardine, Simplical presheaves, J. Pure Appl. Algebra 47(1) (1987), 35-87.
\bibitem{Jar2} J. F. Jardine, Motivic symmetric spectra, Doc. Math. 5 (2000), 445-552.
\bibitem{Kni} A. Knizel, Homotopy invariant presheaves with Kor-transfers, MSc Diploma, St. Petersburg State University, 2012.
\bibitem{Mor} F. Morel, The stable $\bb A^1$-connectivity theorems, K-theory 35 (2006), 1-68.
\bibitem{MV} F. Morel, V. Voevodsky, $\bb A^1$-homotopy theory of schemes,
         Publ. Math. IHES 90 (1999), 45-143.
\bibitem{Sch} S. Schwede, An untitled book project about symmetric spectra, available at www.math.uni-bonn.de/$\sim$schwede (version v3.0/April 2012).
\bibitem{SS} S. Schwede, B. Shipley, Stable model categories are categories of modules, Topology 42(1) (2003), 103-153.
\bibitem{SV1} A. Suslin, V. Voevodsky, Bloch--Kato conjecture and motivic cohomology with
         finite coefficients, The Arithmetic and Geometry of Algebraic Cycles
         (Banff, AB, 1998), NATO Sci. Ser. C Math. Phys. Sci., Vol. 548,
         Kluwer Acad. Publ., Dordrecht (2000), pp. 117-189.
\bibitem{T} R. W. Thomason, T. Trobaugh, Higher algebraic K-theory of schemes and of
         derived categories, The Grothendieck Festschrift~III, Progress in
         Mathematics 88, Birkh\"auser, 1990, pp.~247-435.
\bibitem{Voe} V. Voevodsky, Cohomological Theory of Presheaves with Transfers,
         in Cycles, Transfers and Motivic Homology Theories (V. Voevodsky, A.
         Suslin and E. Friedlander, eds.), Annals of Math. Studies, Princeton University Press, 2000.
\bibitem{Voe1} V. Voevodsky, Triangulated category of motives over a
         field, in Cycles, Transfers and Motivic Homology Theories (V.
         Voevodsky, A. Suslin and E. Friedlander, eds.), Annals of Math.
         Studies, Princeton University Press, 2000.
\bibitem{Voe2} V. Voevodsky, Notes on framed correspondences, unpublished, 2001. Also available at
https://www.math.ias.edu/vladimir/publications
\bibitem{Wal} F. Waldhausen, Algebraic K-theory of spaces, In Algebraic and
         geometric topology, Proc. Conf., New Brunswick/USA 1983, Lecture
         Notes in Mathematics, No.~1126, Springer-Verlag, 1985, pp.~318-419.
\bibitem{Wlk} M. E. Walker, Motivic cohomology and the K-theory of
         automorphisms, PhD Thesis, University of Illinois at Urbana-Champaign, 1996.
\end{thebibliography}
\end{document}